\documentclass{amsart}
\usepackage[margin=2.4cm]{geometry}
\usepackage{xcolor}
\usepackage{graphics, amsmath,amssymb, enumerate, comment, url, mathtools}
\usepackage{graphicx}
\usepackage{epsfig}
\usepackage{hyphenat}
\usepackage{amsmath,amsthm}
\usepackage{amssymb,latexsym}
\usepackage{enumerate}
\usepackage{amssymb}
\usepackage{mathrsfs}
\usepackage{amsfonts}

\newcommand{\R}{\mathbb{R}}

\newcommand{\N}{\mathbb{N}}
\newcommand{\mH}{\mathcal{H}}

\newtheorem{thm}{Theorem}[section]
\newtheorem{lem}[thm]{Lemma}
\newtheorem{prop}[thm]{Proposition}
\newtheorem{cor}[thm]{Corollary}

\theoremstyle{remark}

\usepackage{geometry}
\newcommand\hdim{\dim_H}

\numberwithin{equation}{section}

\begin{document}
\title[Convergence exponent of Dirichlet non-improvable numbers]
 {Convergence exponent of Dirichlet non-improvable numbers in the theory of continued fractions}

\author {Xiaoyan Tan}
\address{School of General Education and International Studies,
Chongqing Polytechnic University of Electronic Technology,
Chongqing, 401331, P. R. China}
\email{miniyan\_tan@163.com}

\author{Zhenliang Zhang$^{*}$}
\address{School of Mathematical Sciences, Chongqing Normal University, Chongqing, 401331, P. R. China}
\email{zhliang\_zhang@163.com}

\thanks {* Corresponding author}

\begin{abstract}\sloppy
Let $x \in [0,1)$ be an irrational number with continued fraction expansion
$[a_1(x),a_2(x), \cdots,a_n(x),\cdots]$
and $q_n(x)$ be the denominator of its $n$-th convergent.
 We establish,
for any $\alpha,\beta$ in $[0,+\infty]$,
the Hausdorff dimension formula of the intersections of the sets of Dirichlet non-improvable numbers
and the level set of convergent exponent,
i.e.
$$
G(\alpha,\beta):
=\left\{x\in[0,1)\colon \tau(x)=\alpha,\,\,\text{and} \,\,
\limsup_{n\to\infty}\frac{\log (a_n(x)a_{n+1}(x))}{\log q_n(x)}\geq\beta\right\},
$$
and
$$
E(\alpha,\beta):
=\left\{x\in[0,1)\colon \tau(x)=\alpha,\,\,\text{and} \,\,
\limsup_{n\to\infty}\frac{\log (a_n(x)a_{n+1}(x))}{\log q_n(x)}=\beta\right\},
$$
where
$$
\tau(x):= \inf\Big\{s \geq 0: \sum_{n \geq 1} a^{-s}_n(x)<\infty\Big\}.
$$

\end{abstract}


 \maketitle

\addtocounter{section}{0}
\section{Introduction}

How well an irrational number can be approximated by rationals is a long standing question in number theory,
which is one of the main topics of research of Diophantine approximation.
Among various kinds of expansions of real number,
continued fraction expansions provide a quick and efficient way to find good rational approximations to real numbers.
Let us begin by recalling that each irrational number $x \in [0,1)$ admits a unique  \emph{continued fraction expansion} of the form
\begin{equation}
\label{CF def}
x = \dfrac{1}{a_1(x) +\dfrac{1}{a_2(x)+\ddots+\dfrac{1}{a_n(x)+\ddots}}}=[a_1(x), a_2(x),\cdots, a_n(x),\cdots],
\end{equation}
where $a_1(x),a_2(x),a_3(x),\cdots$ are positive integers, called the \emph{partial quotients} of the continued fraction expansion of $x$.
The $n$-th convergent $p_n(x)/q_n(x)$ of $x$ is defined by the finite truncation of (\ref{CF def}):
$$
p_n(x)/q_n(x)=[a_1(x), a_2(x),\cdots, a_n(x)].
$$

The convergents give the best rational approximations and the rate of approximation
of the sequence of convergents can be characterized by
$$
\frac{1}{(a_{n+1}(x)+2)q_n^2(x)}
\leq \left|x-\frac{p_n(x)}{q_n(x)}\right|
\leq \frac{1}{a_{n+1}(x)q_n^2(x)}.
$$
This implies that
the asymptotic Diophantine properties of a points $x\in[0,1)$ can be reflected by the growth rate of its partial quotients.

Due to the tight connection between Diophantine approximation and continued fraction expansion,
Hausdorff dimensions of sets of  points obeying some restrictions on their partial quotients
have been computed extensively in recent years.
It seems that the earliest research in this direction can be traced back to Jarn\'ik's paper \cite{Jarnik1928},
in which he studied the set of points whose partial quotients are bounded,
i.e.
the set of badly approximable points from the point of view of Diophantine approximation.
Later on,
Good in \cite{Good} gave an overall investigation of sets with some restrictions on their partial quotients,
including the set $\{x\in [0,1)\colon \lim_{n\to\infty}a_n(x)=+\infty\}$.
For more results of this kind,
one is referred to the work of  Feng, Wu, Liang and Tseng \cite{Fengwuliang},
Hirst \cite{Hirst}, {\L}uczak \cite{Luczak},
and Wang and Wu \cite{Wangwu}.
Especially, in \cite{Wangwu},
let $\phi:\mathbb{N}\to\mathbb{R}^+$ be a positive function,
the authors considered the set
$\{x\in[0,1)\colon a_n(x)\geq \phi(n)\; \text{for infinitely many} \; n\}$
and determined its Hausdorff dimension completely.

Another way of reflection of the growth rate of the partial quotients is the convergence exponent of its sequence.
The convergence exponent of the sequence of partial quotients of $x$ in continued fractions
(see P\'{o}lya and Szeg\H{o} \cite[p.\,26]{PS72}),
given by
\begin{equation}\label{tdy}
\tau(x):= \inf\Big\{s \geq 0: \sum_{n \geq 1} a^{-s}_n(x)<\infty\Big\},
\end{equation}
reflects how fast the growth rate of partial quotients tending to infinity in some sense.
By Borel-Cantelli lemma,
it can be easily verified that
infinitely many of $a_n(x)$ equal to 1 for Lebesgue amlost all $x\in[0,1)$.
This implies that $\tau(x)=\infty$ for Lebesgue almost all $x\in[0,1)$.
The exceptional set $\Delta(\alpha)=\big\{x\in[0,1):\ \tau(x)=\alpha\big\}$ for this property
has been studied by Fang, Ma, Song and Wu \cite[Theorem 2.1]{FMSW21} recently,
and it was proved that for any $0\leq\alpha<\infty$,
$
\hdim \Delta(\alpha)=\frac{1}{2}.
$
%Very recently,
%Fang, Ma, Song and Yang \cite{FMSY24} also considered the convergence exponent of the consecutive partial quotients sequence
%$\{a_n(x)\cdots a_{n+m}(x)\}_{n\geq1}$
%and obtained the Hausdorff dimension of its level sets.
Here and in the sequel,
the notation $\hdim$ means the Hausdorff dimension.
In addition,
the intersection of $\Delta(\alpha)$ and some other fractal sets such as
the set of points with non-decreasing partial quotients \cite{FMSW21},
level sets of the irrational exponent \cite{STZ24} have been studied.
More precisely in \cite{STZ24},
the authors characterized the Hausdorff dimension of the intersection of the set $\Delta(\alpha)$  and the set
$$
\Big\{x\in[0,1)\colon \limsup_{n\to\infty}\frac{\log a_{n+1}(x)}{\log q_n(x)}=\beta\Big\}\ \ (\beta\geq0)
$$
and proved that the Hausdorff dimension of the intersection is the product of their Hausdorff dimensions.

%As far as  we know,
%various exponents related to continued fractions have been recently investigated by many authors.
%To mention a few of them, we refer to Pollicott and Weiss \cite{lesPW99}
% for the Lyapunov exponent of Gauss map,
%Kesseb\"{o}hmer and Stratmann \cite{lesKS} for the Minkowski's question mark function,
%Fan, Liao, Wang and Wu \cite{FLWW09} for the Khintchine exponent,
%Jaffard and Martin \cite{lesJM} for the Brjuno function.

For the uniform Diophantine properties of a points $x\in[0,1)$,
Dirichlet's theorem is a fundamental result.
\begin{thm}[Dirichlet, 1842]
Given $x\in \mathbb{R}$ and $Q>1$,
there exist integers $p,q$ such that
$$
|qx-p|\leq \frac{1}{Q}\ \ \ \text{and}\ \ \ 1\leq q<Q.
$$
\end{thm}
Let $\psi:\mathbb{N}\to \mathbb{R}^+$ be a non-increasing function.
A real number $x\in[0,1)$ is called $\psi$-Dirichlet improvable if for all large enough $Q>1$,
there exist $1\leq q<Q$ and $p\in \mathbb{N}$ such that $|qx-p|<\psi(Q)$.
Denote by $D(\psi)$ the set of $\psi$-Dirichlet improvable numbers.
The elements of the complementary  set $D^c(\psi)$ are
called $\psi$-Dirichlet non-improvable numbers.
For the metric theory of the set $D(t\to\frac{c}{t})$,
Davenprot and Schmidt \cite{DS1970} first proved that
the set $D(t\to\frac{c}{t})$ is of Lebesgue measure zero for any $c<1$ by
showing that $D(t\to\frac{c}{t})$ is a subset of the union of the set of rationals
and the set of irrationals with bounded partial quotients in their continued fraction expansions.
Recently,
Kleinbock and Wadleigh \cite{KW2018} found that
whether a number $x\in D(\psi)$ or not is closely related to
the relative growth of the product of two consecutive partial quotients
 $a_n(x)a_{n+1}(x)$ compared with $q_n(x)$.
More precisely,
suppose that $q\psi(q)<1$ for $q\in \mathbb{N}$,
then
\begin{align*}
\{x\in[0,1)\colon &a_n(x)a_{n+1}(x)\geq \Psi(q_n(x))\,\  \  \text{for}\,\, i.m. \ \ n\in \mathbb{N}\}\subset D^c(\psi)\\
 & \subset\{x\in[0,1)\colon a_n(x)a_{n+1}(x)\geq 4^{-1}\Psi(q_n(x))\,\  \  \text{for}\,\, i.m. \ \ n\in \mathbb{N}\},
\end{align*}
where ``\emph{i.m.}'' stands for infinitely many and
$\Psi(q_n)=q\psi(q)(1-q\psi(q))^{-1}$ for $q\in\mathbb{N}$.
Roughly speaking,
the $\psi$-Dirichlet non-improvable set $D^c(\psi)$,
expressed in terms of continued fractions,
can be written as
$$
G(\Psi):=\{x\in [0,1)\colon a_n(x)a_{n+1}(x)\geq \Psi(q_n(x)) \,\  \  \text{for}\,\, i.m. \ \ n\in \mathbb{N}\}.
$$
The zero-one law for the Lebesgue measure has been established by Kleinbock-Wadleigh \cite{KW2018}.
In Hussain, Kleinbock, Wadleigh and Wang \cite{HKWW2018},
the authors established the zero-infinity law of the Hausdorff measure of $G(\Psi)$ as follows.
\begin{thm}[{\cite{HKWW2018}}, Theorem 1.8]
\label{thmHau}
Let $\Psi:\N\to\R^+$ be a non-decreasing function.
Then for any $0\leq s<1$,
\begin{equation*}
\mH^s(G(\Psi))=
\begin{cases}
0,\ \ \ \ \ \ \ \ \ \ \ \ \ \text{if}\ \ \sum_{t\geq1}t^{1-2s}\Psi(t)^{-s}<\infty,\\
\infty,\ \ \ \ \ \ \ \ \ \ \ \text{if}\ \ \sum_{t\geq1}t^{1-2s}\Psi(t)^{-s}=\infty,
\end{cases}
\end{equation*}
where $\mH^s$ denotes the $s$-dimensional Hausdorff measure.
\end{thm}
For any $\beta\geq 0$,
let
\begin{equation}
\label{G(beta)}
G_{\beta}:=\left\{x\in[0,1)\colon \limsup_{n\to\infty}\frac{\log (a_n(x)a_{n+1}(x))}{\log q_n(x)}\geq\beta\right\},
\end{equation}
and
\begin{equation}
\label{E(beta)}
E_{\beta}:=\left\{x\in[0,1)\colon \limsup_{n\to\infty}\frac{\log (a_n(x)a_{n+1}(x))}{\log q_n(x)}=\beta\right\}.
\end{equation}
Following \cite{Fengxu},
we call $G_{\beta}$ and $E_{\beta}$ the set of Dirichlet non-improvable numbers with order $\beta$ 
and the set of Dirichlet non-improvable numbers with exact order $\beta$ in continued fractions, respectively.
By Theorem \ref{thmHau},
we have
$$
\hdim G_{\beta}=\hdim E_{\beta}=\frac{2}{\beta+2}.
$$
Recently, for $\tau\geq0$ and $0\leq \alpha\leq \beta\leq +\infty$,
the Hausdorff dimension formula of level sets
$$
\left\{x\in[0,1)\colon \lim_{n\to\infty}\frac{\log (a_n(x)a_{n+1}(x))}{\log q_n(x)}=\tau\right\}
$$
and
$$
\left\{x\in[0,1)\colon \liminf_{n\to\infty}\frac{\log (a_n(x)a_{n+1}(x))}{\log q_n(x)}=\alpha,\ \
\limsup_{n\to\infty}\frac{\log (a_n(x)a_{n+1}(x))}{\log q_n(x)}=\beta\right\}
$$
have been established in Huang and Wu \cite{Huangwu}, and Feng and Xu \cite{Fengxu}.

%Multifractal analysis of sets characterized by two (or more) different Diophantine char-
%acteristics potentially could show that they are independent,
%or, conversely,
%it could help to detect profound links between these characteristics.
This paper is concerned with the intersection of  the set of Dirichlet non-improvable numbers
and the level sets generated by \eqref{tdy}.
To be precise,
we are interested in the Hausdorff dimension of
$$
G(\alpha,\beta):
=\left\{x\in[0,1)\colon \tau(x)=\alpha \,\,\text{and} \,\,
\limsup_{n\to\infty}\frac{\log (a_n(x)a_{n+1}(x))}{\log q_n(x)}\geq\beta\right\},
$$
and
$$
E(\alpha,\beta):
=\left\{x\in[0,1)\colon \tau(x)=\alpha \,\,\text{and} \,\,
\limsup_{n\to\infty}\frac{\log (a_n(x)a_{n+1}(x))}{\log q_n(x)}=\beta\right\}.
$$
Now, we are in a position to state the main results.
\begin{thm}\label{mainthm}
For any $0\leq\alpha\leq\infty$ and $\beta\geq0$, we have
\begin{equation*}
\hdim G(\alpha,\beta)=\hdim E(\alpha,\beta)=
\begin{cases}
\frac{2}{\beta+2},\ \ \ \ \ \ \ \ \ \ \ \ \ \, \ \alpha=\infty ;\cr
\frac{2}{\beta+2+\sqrt{\beta^2+4}},\ \ \ \ \ \ 0\leq\alpha<\infty.
\end{cases}
\end{equation*}
\end{thm}

As a byproduct of Theorem \ref{mainthm},
we obtain an extension of the following Theorem \ref{dimthm1},
proved by Feng and Xu \cite{Fengxu}.
For any $\beta\geq0$,
let
$$
G_{\infty}(\beta)=\left\{x\in[0,1)\colon \lim_{n\to\infty}a_n(x)=+\infty \,\,\text{and} \,\,
\limsup_{n\to\infty}\frac{\log (a_n(x)a_{n+1}(x))}{\log q_n(x)}\geq\beta\right\}
$$
and
$$
E_{\infty}(\beta)=\left\{x\in[0,1)\colon \lim_{n\to\infty}a_n(x)=+\infty \,\,\text{and} \,\,
\limsup_{n\to\infty}\frac{\log (a_n(x)a_{n+1}(x))}{\log q_n(x)}=\beta\right\}.
$$
\begin{cor}
For any $\beta\geq0$,
we have
$\hdim G_{\infty}(\beta)=\hdim E_{\infty}(\beta)=\frac{2}{\beta+2+\sqrt{\beta^2+4}}$.
\end{cor}
The paper is organized as follows.
In Section 2,
we first collect some elementary properties
and then present the mass distribution principle
for computing the lower bound of the Hausdorff dimensions of fractal sets.
Sections 3 is devoted to the proofs of our main results.

We use $\mathbb{N}$ to denote the set of all positive integers,
$|\cdot|$ the length of a subinterval of $[0,1)$ and $\sharp$ the cardinality of a set, respectively.
\section{Preliminaries}
In this section,
we first collect some notations and basic properties and then present some useful lemmas for
calculating the Hausdorff dimension of sets in continued fractions. \\
\indent  For any $n\geq1$ and $(a_1,\cdots,a_n)\in\mathbb{N}^{n}$, we call
\begin{equation*}
I_n(a_1, \cdots, a_n): =\left\{x\in[0,1):\ a_1(x)=a_1, \cdots, a_n(x)=a_n\right\}
\end{equation*}
a \emph{basic interval of order} $n$ of continued fractions.
It is worth pointing out that all points in $I_n(a_1, \cdots, a_n)$ have a continued fraction expnsion
beginning with $a_1,\cdots,a_n$ and thus the same $p_n(x)$ and $q_n(x)$.
If there is no confusion, we write
$p_n(a_1,\cdots,a_n)=p_n=p_n(x)$ and $\ q_n(a_1,\cdots,a_n)=q_n=q_n(x)$.
It is well known (see \cite[p.4]{Khi64}) that $p_n$ and $q_n$ satisfy the following recursive formula:
\begin{equation}\label{ppqq}
\begin{cases}
p_{-1}=1,\ \ p_0=0,\ \ p_n=a_np_{n-1}+p_{n-2}\ (n\geq1);\cr
q_{-1}=0,\ \ \ q_0=1,\ \ q_n=a_nq_{n-1}+q_{n-2}\ (n\geq1).
\end{cases}
\end{equation}
As consequences, we have the following results.
\begin{prop}[{\cite[p. 18]{IK02}}]\label{cd}
For any $(a_1,\cdots, a_n)\in\mathbb{N}^{n}$, the interval $I_n(a_1,\cdots, a_n)$ has the endpoints
$p_n/q_n$ and $(p_n+p_{n-1})/(q_n+q_{n-1})$. More precisely,
\begin{equation*}
I_n(a_1,\cdots, a_n)=
\begin{cases}
[\frac{p_n}{q_n},\frac{p_n+p_{n-1}}{q_n+q_{n-1}}),\ \ \ \text{if}\ n\ \text{is even},\cr
(\frac{p_n+p_{n-1}}{q_n+q_{n-1}},\frac{p_n}{q_n}],\ \ \ \text{if}\ n\ \text{is odd}.
\end{cases}
\end{equation*}
As a result,
\begin{equation*}
|I_n(a_1, \cdots, a_n)|=\frac{1}{q_n(q_n+q_{n-1})}.
\end{equation*}
\end{prop}
\begin{lem}[{\cite[p. 13]{Khi64}}]\label{qnxx}
For any $(a_1,\cdots, a_n)\in\mathbb{N}^{n}$, we have
\[
q_n\ge2^{\frac{n-1}{2}}\ \ \text{and}\ \ \prod_{k=1}^{n}a_k\leq q_n\leq\prod_{k=1}^{n}(a_k+1).
\]
\end{lem}
\begin{lem}[{\cite[Lemma 2.1]{Wu06}}]\label{sqn}
For any $n\ge1$ and $1\le k\le n$, we have
$$
\dfrac{a_k+1}{2}\le \frac{q_n(a_1,\cdots,a_{k-1},a_{k},a_{k+1},\cdots,a_n)}{q_{n-1}(a_1,\cdots,a_{k-1},a_{k+1},\cdots,a_n)}\le a_k+1.
$$
\end{lem}

For any $m\in\mathbb{N}$, let
$B_m=\{x\in[0,1)\colon 1\le a_n(x)\le m\;\;\text{for any }\;n\ge1\}$.
Jarn\'ik calculated the Hausdorff dimension of $B_m$.
\begin{lem}
[\cite{Jarnik1928}]
\label{bounded}
For any $m\ge 8$,
$$
1-\dfrac{1}{m\log2}\le\hdim B_m\le 1-\dfrac{1}{8m\log m}.
$$
In particular, the set
$$
B=\{x\in[0,1)\colon \sup_{n\ge1}a_n(x)<+\infty\}
$$
is of Hausdorff dimension 1.
\end{lem}

The following lemma, established by Feng and Xu \cite{Fengxu},
is concerned with the Hausdorff dimension of the intersection of the set of Dirichlet non-improvable numbers
and the set of points whose L\'{e}vy constant equal to infinity in continued fractions.
\begin{thm}[\cite{Fengxu}]
\label{dimthm1}
For any $\beta>0$, we have
\[
\hdim \Big\{x\in G_{\beta}\colon \lim_{n\to\infty}\frac{\log q_n(x)}{n}=\infty\Big\}
=\frac{2}{\beta+2+\sqrt{\beta^2+4}},
\]
where the set $ G_{\beta}$ is defined as in \eqref{G(beta)}.
\end{thm}
To end this section, we present the mass distribution principle,
 which is usually applied to obtain the lower bound of the Hausdorff dimension of a set.
\begin{lem}[{\cite[Proposition 2.3]{Fal97}}]\label{mass principle}
Let $E\subseteq[0,1)$ be a Borel set and let $\mu$ be a finite measure with $\mu(E)>0$.
If
  \begin{equation*}
    \liminf\limits_{r\to 0}\frac{\log\mu(B(x,r))}{\log r}\geq s\ \ \text{for all}\ x\in E,
  \end{equation*}
where $B(x,r)$ denotes the open ball with center $x$ and radius $r$,
then we have $\hdim E\geq s$.
\end{lem}
\section{Proof of the main results}
This section is devoted to the proofs of the main results.
Recall that
$$
G(\alpha,\beta)
=\left\{x\in[0,1)\colon \tau(x)=\alpha \,\,\text{and} \,\,
\limsup_{n\to\infty}\frac{\log (a_n(x)a_{n+1}(x))}{\log q_n(x)}\geq\beta\right\},
$$
and
$$
E(\alpha,\beta)
=\left\{x\in[0,1)\colon \tau(x)=\alpha \,\,\text{and} \,\,
\limsup_{n\to\infty}\frac{\log (a_n(x)a_{n+1}(x))}{\log q_n(x)}=\beta\right\}.
$$
To prove Theorem \ref{mainthm},
we shall give the upper bound of Hausdorff dimension of $ G(\alpha,\beta)$
and the lower bound of Hausdorff dimension of $ E(\alpha,\beta)$ respectively.

If $\beta=0$ and $0\leq \alpha<\infty$,
on one hand,
it is obvious that
$$
G(\alpha,0)\subset \big\{x\in[0,1):\ \tau(x)=\alpha\big\},
$$
which shows that $\hdim G(\alpha,0)\leq \frac{1}{2}$.
On the other hand,
the Cantor type sets constructed in the proof of Theorem 2.1 in \cite{FMSW21} are also
subsets of $E(\alpha,\beta)$.
Then, we have $\hdim G(\alpha,0)= \frac{1}{2}$.

If $\beta=0$ and $\alpha=\infty$,
it is obvious that the set $B$ defined in Lemma \ref{bounded} is a subset of $E(\infty,0)$.
By Lemma \ref{bounded},
it holds that $\hdim G(\infty,0)= \hdim E(\infty,0)= 1$.

If $\beta>0$ and $\alpha=\infty$,
it is easy to see that the Cantor type sets constructed in the proof of Theorem 3 in \cite{Bugeaud2011}
is a subset of of $E(\infty,\beta)$ where we take $\Psi(x)=x^{-2-\beta}$.
Then we have $\hdim E(\infty,\beta)\geq \frac{2}{\beta+2}$.
Since $E(\infty,\beta)\subset E_{\beta}$,
it holds that $\hdim E(\infty,\beta)=\frac{2}{\beta+2}$.

In what follows,
we always assume that $\beta>0$ and $0\leq \alpha<\infty$.
\subsection{Upper bound:}

For any $x\in G(\alpha,\beta)$,
we have $\tau(x)=\alpha$.
Hence, for any $\varepsilon>0$,
we deduce from the definition of
$\tau(x)$ that $\sum_{n\geq1}a_n(x)^{-(\alpha+\varepsilon)}<\infty$.
Thus $a_n(x)\to\infty$ as $n\to\infty$.
By Lemma \ref{qnxx} and the general form of the Stolz-Ces\`{a}ro theorem,
which states that if $\{\xi_n\}_{n\geq1}$ and $\{\eta_n\}_{n\geq1}$ are two sequences
such that $\{\eta_n\}_{n\geq1}$ is monotone and unbounded,
then
\begin{equation*}
\liminf_{n\to\infty}\dfrac{\xi_{n+1}-\xi_n}{\eta_{n+1}-\eta_n}
\leq \liminf_{n\to\infty}\dfrac{\xi_n}{\eta_n}
\leq \limsup_{n\to\infty}\dfrac{\xi_n}{\eta_n}
\leq \limsup_{n\to\infty}\dfrac{\xi_{n+1}-\xi_n}{\eta_{n+1}-\eta_n},
\end{equation*}
we have
$$
\liminf_{n\to\infty}\dfrac{\log q_n(x)}{n}
\geq \liminf_{n\to\infty}\dfrac{\log (a_1(x)a_2(x)\cdots a_n(x))}{n}
\geq \liminf_{n\to\infty} \log a_n(x)=\infty.
$$
Hence, we obtain
$$
G(\alpha,\beta)\subseteq \Big\{x\in G_{\beta}\colon \lim_{n\to\infty}\frac{\log q_n(x)}{n}=\infty\Big\}.
$$
By Lemma \ref{dimthm1},
we conclude that
$$
\hdim  G(\alpha,\beta)\leq\frac{2}{\beta+2+\sqrt{\beta^2+4}}.
$$
\subsection{Lower bound for the case $0<\alpha<\infty$}

To bound $\hdim E(\alpha,\beta)$ from below,
we adopt a method, applied successfully in \cite{STZ24} and \cite{ZhangLT},
that consists in inserting a suitable sequence into continued fraction expansions to
construct points in $E(\alpha,\beta)$.
\medskip

The points in $E(\alpha,\beta)$ could be constructed by three steps.
First,
let $N_0=\lfloor2^{\alpha}\rfloor+1$,
and
let
$$
F(\alpha)
=\Big\{x\in[0,1)\colon \lfloor(N_0+n)^{1/\alpha}\rfloor+1
\leq a_n(x)\leq2\lfloor(N_0+n)^{1/\alpha}\rfloor,\,\forall \,n\geq 1
\Big\}.
$$
Second,
let $\{n_j\}_{j\geq1}$  be a sequence of positive integers such that
$$
n_1=\max\left\{N_0,\lfloor(N_0+10)^{\frac{2N_0+20}{\alpha}}\rfloor\right\}+1\ \
\text{and}\ \ n_{j+1}=n_j^j\ (\forall\ j\geq1).
$$
Using the sequence $\{n_j\}_{j\geq1}$,
we then define a sequence $\{c_j\}_{j\geq1}$ and $\{e_j\}_{j\geq1}$ satisfying
\begin{equation}\label{cj}
 \lfloor n^{\lambda-1}_j\rfloor+1\leq c_j\leq2\lfloor n^{\lambda-1}_j\rfloor\ (\forall\ j\geq1)
 \end{equation}
and
\begin{equation}\label{ej}
 \lfloor n^{\lambda^2-\lambda}_j\rfloor+1\leq e_j\leq2\lfloor n^{\lambda^2-\lambda}_j\rfloor\ (\forall\ j\geq1)
 \end{equation}
 where $\lambda-\frac{1}{\lambda}=\beta$ and $\lambda=\frac{\beta+\sqrt{\beta^2+4}}{2}$.

Third,
we insert the sequence $\{c_j\}_{j\geq1}$ and $\{e_j\}_{j\geq1}$ into the continued fraction expansion of
 $x=[a_1,a_2,\cdots]\in F(\alpha)$ by putting
\begin{align*}
y=f(x)
&=[b_1,b_2,\cdots,b_n,\cdots]\\
&=[a_1,a_2,\cdots, a_{r_1},c_1,e_1,a_{r_1+1},\cdots,a_{r_{j-1}},c_{j-1},e_{j-1},a_{r_{j-1}+1},\cdots],
\end{align*}
where $r_j$ is the smallest $r$ such that
\begin{equation}
\label{zxrj}
q_{r+2j-2}(a_1,\cdots, a_{r_1},c_1,e_1,a_{r_1+1},\cdots,a_{r_{j-1}},c_{j-1},e_{j-1},a_{r_{j-1}+1},\cdots,a_r)
\geq n_j.
\end{equation}
Collect all these points $y$ and let
\[
A(\alpha,\beta)=\big\{y=f(x)\in(0,1)\colon x\in F(\alpha)\big\}.
\]
\begin{lem}\label{bhzj}
For any $0<\alpha<\infty$ and $\beta>0$, we have $A(\alpha,\beta)\subseteq E(\alpha,\beta)$.
\end{lem}
\begin{proof}
From the definition of $E(\alpha,\beta)$,
we  need only to  prove that
\[\tau(y)=\alpha\ \ \text{and}\ \ \limsup\limits_{n\to\infty}\frac{\log (b_n(y)b_{n+1}(y))}{\log q_n(y)}=\beta\]
for any $y\in A(\alpha,\beta)$.
For this purpose,
we shall give some estimates.
By the smallest property of $r_j$ in \eqref{zxrj},
Lemma \ref{qnxx} and Stirling's formula which states that
$$\sqrt{2\pi}n^{n+\frac{1}{2}}e^{-n}\leq n!\leq en^{n+\frac{1}{2}}e^{-n},$$
we have
\begin{align*}
  n_j
&>q_{r_j+2j-3}(a_1,a_2,\cdots, a_{r_1},c_1,e_1,a_{r_1+1},\cdots,a_{r_j-1})\\
&\geq\prod_{k=1}^{r_j-1}a_k\geq ((r_j-1)!)^{\frac{1}{\alpha}}\geq (\frac{r_j-1}{e})^{\frac{r_j-1}{\alpha}}.
\end{align*}
This shows  that
\begin{equation*}
\log n_j> \frac{r_j-1}{\alpha}(\log (r_j-1)-1)\geq \frac{r_j-1}{2\alpha}\geq \frac{r_j}{4\alpha}.
\end{equation*}
Combining this inequality and Lemma \ref{sqn}, we have
\begin{align}\label{xyqj}
\nonumber
&\ q_{r_j+2j-2}(a_1,a_2,\cdots, a_{r_1},c_1,e_1,a_{r_1+1},\cdots,a_{r_{j-1}},c_{j-1},e_{j-1},a_{r_{j-1}+1},\cdots,a_{r_j})\\
&\leq (a_{r_j}+1)n_j\leq4(8\alpha\log n_j+N_0)^{\frac{1}{\alpha}} n_j.
\end{align}
For any $r_j+2j+1\leq n\leq r_{j+1}+2j$ for $j\geq1$,
we have $b_{n}(y)=a_i(x)$ for some $r_j+1\leq i\leq r_{j+1}$.
Then
$$
b_{n}(y)\leq a_{r_{j+1}}(x)\leq 2(8j\alpha\log n_j+N_0)^{\frac{1}{\alpha}},
$$
when
$q_n\geq n_j$.
By the choice of sequence $n_j$,
$\log (b_{n}(y)b_{n+1}(y))/\log q_n(y)$ can be arbitrarily small
for any $r_j+2j+1\leq n\leq r_{j+1}+2j-1$ when $j$ is large enough.
For $n=r_j+2j-2$ for some $j\geq1$,
we have
\begin{align*}
\lim\limits_{j\to\infty}\frac{\log( b_{r_j+2j-2}(y)b_{r_j+2j-1}(y))}{\log q_{r_j+2j-2}(y)}
=\lim\limits_{j\to\infty}\frac{\log a_{r_j}(x)c_j}{\log n_j}=\lambda-1.
\end{align*}

For $n=r_j+2j$ for some $j\geq1$,
we have
\begin{align*}
\lim\limits_{j\to\infty}\frac{\log( b_{r_j+2j}(y)b_{r_j+2j+1}(y))}{\log q_{r_j+2j}(y)}
=\lim\limits_{j\to\infty}\frac{\log e_j a_{r_j+1}(x)}{\log q_{r_j+2j}(y)}
=\lim\limits_{j\to\infty}\frac{(\lambda^2-\lambda)\log n_j}{\lambda^2\log n_j}=1-\frac{1}{\lambda}.
\end{align*}

Thus, it follows from \eqref{zxrj}, \eqref{xyqj}
and the above inequalities that
\begin{align*}
\limsup_{n\to \infty}\frac{\log (b_{n}(y)b_{n+1}(y))}{\log q_n(y)}
&=\lim\limits_{j\to\infty}\frac{\log (b_{r_j+2j-1}(y)b_{r_j+2j}(y))}{\log q_{r_j+2j-1}(y)}
=\lim\limits_{j\to\infty}\frac{\log c_je_j}{\lambda\log n_j}\\
&=\lim\limits_{j\to\infty}\frac{(\lambda^2-1)\log n_j}{\lambda\log n_j}
=\frac{\lambda^2-1}{\lambda}=\beta.
\end{align*}

In the following, we shall prove that $\tau(y)=\alpha$.
In fact,
by the construction of $F(\alpha)$,
it is equivalent  to prove that for any $\varepsilon>0$,
the series $\sum_{j\geq1}c^{-(\alpha+\varepsilon)}_j$
and $\sum_{j\geq1}e^{-(\alpha+\varepsilon)}_j$ converge.
By \eqref{xyqj},
for any $\varepsilon>0$,
we have $q_{r_j+2j-2}\leq n_j^{1+\varepsilon}$
for $j$ sufficiently large.
Then by Lemma \ref{qnxx}, \eqref{cj} and \eqref{ej},
there exists some constant $C_0$ such that
\begin{align*}
\sum\limits_{j\geq1}c_j^{-(\alpha+\varepsilon)}+\sum\limits_{j\geq1}e_j^{-(\alpha+\varepsilon)}
&\leq\sum\limits_{j\geq1}n_j^{-(\alpha+\varepsilon)(\lambda-1)}
       +\sum\limits_{j\geq1}n_j^{-(\alpha+\varepsilon)(\lambda^2-\lambda)}\\
&\leq\sum\limits_{j\geq1}q_{r_j+2j-2}^{\frac{-(\alpha+\varepsilon)(\lambda-1)}{1+\varepsilon}}
       +\sum\limits_{j\geq1}q_{r_j+2j-2}^{\frac{-(\alpha+\varepsilon)(\lambda^2-\lambda)}{1+\varepsilon}}+C_0\\
&\leq\sum\limits_{j\geq1}2^{\frac{r_j+2j-2}{2}\cdot\frac{-(\alpha+\varepsilon)(\lambda-1)}{1+\varepsilon}}
       +\sum\limits_{j\geq1}2^{\frac{r_j+2j-2}{2}\cdot\frac{-(\alpha+\varepsilon)(\lambda^2-\lambda)}{1+\varepsilon}}+C_0\\
&\leq\sum\limits_{j\geq1}2^{(j-1)\cdot\frac{-(\alpha+\varepsilon)(\lambda-1)}{1+\varepsilon}}
       +\sum\limits_{j\geq1}2^{(j-1)\cdot\frac{-(\alpha+\varepsilon)(\lambda^2-\lambda)}{1+\varepsilon}}+C_0<\infty.
\end{align*}
\end{proof}

Now, we turn to estimate the lower bound of $\hdim  A(\alpha,\beta)$.
To do this,
we will analyze the structure of $A(\alpha,\beta)$,
construct a measure supported on $A(\alpha,\beta)$,
then use the mass distribution principle to get
the lower bound of the Hausdorff dimension of $A(\alpha,\beta)$.
\subsubsection{\textbf{Structure of $A(\alpha,\beta)$}}
To illustrate the structure of $A(\alpha,\beta)$,
we shall make use of a kind of symbolic space as follows. For simplicity, denoted by
$m_0=-2$ and $m_j=r_j+2(j-1)\ (\forall\ j\geq1)$.
For any $n\geq 1$, let
\begin{align*}
D_n(\alpha,\beta)=
\Big\{(\sigma_1,\cdots,\sigma_n)&\in \mathbb{N}^n\colon \lfloor(N_0+k-2j)^{\frac{1}{\alpha}}\rfloor+1
  \leq \sigma_k\leq2\lfloor(N_0+k-2j)^{\frac{1}{\alpha}}\rfloor\,\,\,\text{for}\,\,\,\\
& m_j+2<k\leq m_{j+1}\ (\forall j\geq0)\ \text{and}\ \lfloor n_j^{\lambda-1 }\rfloor+1
  \leq \sigma_{m_j+1}\leq2 \lfloor n_j^{\lambda-1 }\rfloor,\,\,\,\\
& \lfloor n_j^{\lambda^2-\lambda}\rfloor+1
  \leq \sigma_{m_j+2}\leq2 \lfloor n_j^{\lambda^2-\lambda}\rfloor\ (\forall j\geq1)
\Big\}
\end{align*}
and
$$
D(\alpha,\beta)=\bigcup_{n=0}^{\infty}D_n(\alpha,\beta),\,\,\,\,(D_0(\alpha,\beta):=\emptyset).
$$
Then for any $(\sigma_1,\cdots,\sigma_n)\in D_n(\alpha,\beta)$,
we call
\begin{equation*}
J_n(\sigma_1,\cdots,\sigma_n)=\bigcup_{\sigma_{n+1}}\emph{cl}\big(I_{n+1}(\sigma_1,\cdots,\sigma_n,\sigma_{n+1})\big)
\end{equation*}
a \emph{fundamental interval of order} $n$,
where  'cl' denotes the closure of a set and the union is taken over all $\sigma_{n+1}$ such that
$(\sigma_1,\cdots,\sigma_n,\sigma_{n+1})\in D_{n+1}(\alpha,\beta)$.
It is easy to see that
\begin{equation*}
A(\alpha,\beta)=\bigcap_{n\geq1}\bigcup_{(\sigma_1,\cdots,\sigma_n)\in D_n(\alpha,\beta)}J_n(\sigma_1,\cdots,\sigma_n).
\end{equation*}
By \eqref{ppqq} and Proposition \ref{cd}, we shall estimate the length of the fundamental interval of order $n$.
If $m_j+1<n<m_{j+1}$  for any $j\geq 1$,
\begin{align}\label{jnj}
\nonumber
&\ |J_n(\sigma_1,\cdots,\sigma_n)|\\
\nonumber&=\sum\limits_{\lfloor(N_0+n+1-2j)^{\frac{1}{\alpha}}\rfloor+1
               \leq \sigma_{n+1}\leq2\lfloor(N_0+n+1-2j)^{\frac{1}{\alpha}}\rfloor}
                                   |I_{n+1}(\sigma_1,\cdots,\sigma_n,\sigma_{n+1})|\\
\nonumber&=\sum\limits_{\lfloor(N_0+n+1-2j)^{\frac{1}{\alpha}}\rfloor+1
            \leq \sigma_{n+1}\leq2\lfloor(N_0+n+1-2j)^{\frac{1}{\alpha}}\rfloor}
\frac{1}{(\sigma_{n+1}q_n+q_{n-1})((\sigma_{n+1}+1)q_n+q_{n-1})}\\
&=\frac{\lfloor(N_0+n+1-2j)^{\frac{1}{\alpha}}\rfloor}
{\left((\lfloor(N_0+n+1-2j)^{\frac{1}{\alpha}}\rfloor+1)q_n+q_{n-1}\right)
\left((2\lfloor(N_0+n+1-2j)^{\frac{1}{\alpha}}\rfloor+1)q_n+q_{n-1}\right)}.
\end{align}
If $n=m_j$ or  $n=m_j+1$ for some $j\geq 1$,
by the similar computation as above,
we have
\begin{equation}\label{eq2}
|J_{m_j}(\sigma_1,\cdots,\sigma_{m_j})|=\frac{\lfloor n_j^{\lambda-1 }\rfloor}
{\left((\lfloor n_j^{\lambda-1 }\rfloor+1)q_{m_j}+q_{{m_j}-1}\right)
\left((2\lfloor n_j^{\lambda-1 }\rfloor+1)q_{m_j}+q_{{m_j}-1}\right)}
\end{equation}
and
\begin{equation}\label{eq2-1}
|J_{m_j+1}(\sigma_1,\cdots,\sigma_n)|=\frac{\lfloor n_j^{\lambda^2-\lambda}\rfloor}
{\left((\lfloor n_j^{\lambda^2-\lambda}\rfloor+1)q_{m_j+1}+q_{m_j}\right)
\left((2\lfloor n_j^{\lambda^2-\lambda}\rfloor+1)q_{m_j+1}+q_{m_j}\right)}.
\end{equation}

%%========================================================================================================
\subsubsection{\textbf{A measure supported on $A(\alpha,\beta)$}}

Let $\mu\colon \{J(\sigma),\sigma\in D(\alpha,\beta)\}\to \mathbb{R}^+$ be a set function defined as follows.
For any $(\sigma_1,\cdots,\sigma_n)\in D_n(\alpha,\beta)$,
\[
\mu (J_n(\sigma_1,\cdots,\sigma_n))=\frac{1}{\sharp D_n(\alpha,\beta)}.
\]
More precisely, if $m_j+2<n\leq m_{j+1}$ for any $j\geq1$,
\begin{equation}\label{udy2}
\mu (J_n(\sigma_1,\cdots,\sigma_n))
=\frac{1}{\lfloor(N_0+n-2j)^{\frac{1}{\alpha}}\rfloor}\mu (J_{n-1}(\sigma_1,\cdots,\sigma_{n-1})),
\end{equation}
if $n=m_j+1$ for some $j\geq1$,
\begin{equation}\label{udy3}
\mu (J_{n}(\sigma_1,\cdots,\sigma_n))
=\frac{1}{\lfloor n_j^{\lambda-1 }\rfloor}\mu (J_{n-1}(\sigma_1,\cdots,\sigma_{n-1})),
\end{equation}
and if $n=m_j+2$ for some $j\geq1$,
\begin{equation}\label{udy3}
\mu (J_{n}(\sigma_1,\cdots,\sigma_n))
=\frac{1}{\lfloor n_j^{\lambda^2-\lambda}\rfloor}\mu (J_{n-1}(\sigma_1,\cdots,\sigma_{n-1})).
\end{equation}
It is easy to check that all the fundamental intervals $J_{n}(\sigma_1,\cdots,\sigma_{n})$ generate a semi-algebra.
Note that for $n\geq1$ and $(\sigma_1,\cdots,\sigma_n)\in D_n(\alpha,\beta)$,
\[\sum_{(\sigma_1,\cdots,\sigma_n)\in D_n(\alpha,\beta)}\mu (J_{n}(\sigma_1,\cdots,\sigma_{n}))=1,\]
and
\[
\mu (J_n(\sigma_1,\cdots,\sigma_n))
=\sum_{(\sigma_1,\cdots,\sigma_n,\sigma_{n+1})\in D_{n+1}(\alpha,\beta)}\mu (J_{n+1}(\sigma_1,\cdots,\sigma_{n+1})).
\]
Then the set function $\mu$ is well-defined on the semi-algebra.
Thus by Hahn-Kolmogorov extension theorem (see \cite[Theorem 11.3]{Billingsley}),
we know that the set function $\mu$ can be extended into a probability measure supported on $A(\alpha,\beta)$,
which is still denoted by $\mu$.

%%========================================================================================================
\subsubsection{\textbf{Estimation of $\mu (J_n(\sigma_1,\cdots,\sigma_n))$}}
For each $(\sigma_1,\cdots,\sigma_n)\in D_n(\alpha,\beta)$,
in the following we shall give the estimation of $\mu (J_n(\sigma_1,\cdots,\sigma_n))$.
Fix $0<t<\frac{1}{\lambda+1}=\frac{2}{\beta+2+\sqrt{\beta^2+4}}$ and let $\varepsilon$ be small such that
\[0<\varepsilon<\min\Big\{\frac{1-(\lambda+1)t}{3t+1},\ \frac{1}{2}-t\Big\}.\]
Based on Lemma \ref{qnxx} and Stirling's formula,
we can choose $j_0$ sufficiently large such that
$q_{r_j+2j-2}\leq n_j^{1+\varepsilon}$ for any $j\geq j_0$
and
for any $n\geq m_{j_0}$, $m_j+1\leq n\leq m_{j+1}$ for some $j\geq j_0$,
\begin{align}\label{ineq1}
\nonumber q^{\varepsilon}_n(\sigma_1,\cdots,\sigma_n)
&\geq\Big(\big(\lfloor (N_0+n-2j)^{\frac{1}{\alpha}}\rfloor\big)!\Big)^{\varepsilon}\\
&\geq\big(\frac{N_0+n-2j}{e}\big)^{(N_0+n-2j)\cdot\frac{1}{2\alpha}\varepsilon}
\geq2^{n+3}\lfloor (N_0+n-2j)^{\frac{1}{\alpha}}\rfloor,
\end{align}
where the last inequality holds since $n-2j\geq r_j-1$ and $r_j$ tends to infinity.
In view of \eqref{zxrj} and \eqref{xyqj},
we have
\begin{equation}\label{ineq2}
n_j\leq q_{m_j}=q_{r_j+2j-2}<n_j^{1+\varepsilon}.
\end{equation}
If $m_j+1<n<m_{j+1}$ for any $j\geq1$, by Lemma \ref{qnxx}, \eqref{jnj} and \eqref{ineq1},
\begin{align}\label{ineq-est-mea1}
\nonumber \mu (J_n(\sigma_1,\cdots,\sigma_n))
&=\prod_{k=1}^{n-2j}\frac{1}{\lfloor(N_0+k)^{\frac{1}{\alpha}}\rfloor}
   \cdot\prod_{i=1}^{j}\frac{1}{\lfloor n_i^{\lambda-1 }\rfloor}
   \cdot\prod_{i=1}^{j}\frac{1}{\lfloor n_i^{\lambda^2-\lambda }\rfloor}\leq \frac{2^n}{q_n}\\
&\leq \left(\frac{1}{8\lfloor(N_0+n-2j)^{\frac{1}{\alpha}}\rfloor q_n^2}\right)^{\frac{1}{2}-\varepsilon}
    \leq |J_n(\sigma_1,\cdots,\sigma_n)|^{t}.
\end{align}
If $n=m_j$ for some $j\geq 1$,
then by Lemma \ref{qnxx}, \eqref{eq2}, \eqref{ineq1} and \eqref{ineq2},
\begin{align}\label{ineq-est-mea2}
\nonumber \mu (J_n(\sigma_1,\cdots,\sigma_n))
&=\prod_{k=1}^{n-2j+2}\frac{1}{\lfloor(N_0+k)^{\frac{1}{\alpha}}\rfloor}
   \cdot\prod_{i=1}^{j-1}\frac{1}{\lfloor n_i^{\lambda-1 }\rfloor}
   \cdot\prod_{i=1}^{j-1}\frac{1}{\lfloor n_i^{\lambda^2-\lambda }\rfloor}
   \leq \frac{2^{m_j}}{q_{m_j}}\leq \frac{1}{n_j^{1-\varepsilon}}\\
&\leq \left(\frac{1}{n_j^{\lambda+1+3\varepsilon}}\right)^{t}
\leq\left(\frac{1}{32\lfloor n_j^{\lambda-1}\rfloor q_{m_j}^2}\right)^{t}
\leq |J_n(\sigma_1,\cdots,\sigma_n)|^{t}.
\end{align}

If $n=m_j+1$ for some $j\geq 1$,
then by Lemma \ref{qnxx}, \eqref{eq2-1}, \eqref{ineq1} and \eqref{ineq2},
\begin{align}\label{ineq-est-mea2}
\nonumber \mu (J_n(\sigma_1,\cdots,\sigma_n))
&=\prod_{k=1}^{n-2j+1}\frac{1}{\lfloor(N_0+k)^{\frac{1}{\alpha}}\rfloor}
   \cdot\prod_{i=1}^{j}\frac{1}{\lfloor n_i^{\lambda-1 }\rfloor}
   \cdot\prod_{i=1}^{j-1}\frac{1}{\lfloor n_i^{\lambda^2-\lambda }\rfloor}
   \leq \frac{2^{m_j+1}}{q_{m_j+1}}\leq \frac{1}{n_j^{\lambda-\varepsilon}}\\
&\leq \left(\frac{1}{n_j^{\lambda^2+\lambda+3\varepsilon}}\right)^{t}
\leq\left(\frac{1}{32\lfloor n_j^{\lambda^2-\lambda}\rfloor q_{m_j+1}^2}\right)^{t}
\leq |J_n(\sigma_1,\cdots,\sigma_n)|^{t}.
\end{align}
\subsubsection{\textbf{Estimation of $\mu (B(y,r))$}}

%\subsubsection{\textbf{Gaps in} $A(\alpha,\beta)$}
Before the estimation of measure of any small ball,
we shall estimate the gaps between the adjoint fundamental intervals with the same order.
We assume without loss of generality that $n$ is even,
since the process for $n$ being odd runs similarly.\\
\indent Let $g^r(\sigma_1,\cdots,\sigma_n)$ be the distance between $J_n(\sigma_1,\cdots,\sigma_n)$
and the fundamental interval of order $n$
which is closest to and lies on the right of $J_n(\sigma_1,\cdots,\sigma_n)$,
$g^l(\sigma_1,\cdots,\sigma_n)$ be the distance between $J_n(\sigma_1,\cdots,\sigma_n)$
and the fundamental interval of order $n$
which lies on the left of $J_n(\sigma_1,\cdots,\sigma_n)$  and is closest to it.
Denote
\[\hat{g}(\sigma_1,\cdots,\sigma_n)=\min\{g^r(\sigma_1,\cdots,\sigma_n),g^l(\sigma_1,\cdots,\sigma_n)\}.\]

By the structure of $A(\alpha,\beta)$,
$g^r(\sigma_1,\cdots,\sigma_n)$ is just the distance between the right endpoint of $J_n(\sigma_1,\cdots,\sigma_n)$ and
the right endpoint of $I_n(\sigma_1,\cdots,\sigma_n)$,
and
$g^l(\sigma_1,\cdots,\sigma_n)$ is just the distance between the left endpoint of $J_n(\sigma_1,\cdots,\sigma_n)$ and
the left endpoint of $I_n(\sigma_1,\cdots,\sigma_n)$.
\begin{itemize}
  \item (i)
If $m_j+1<n<m_{j+1}$  for any $j\geq 1$,
then we have
\begin{align*}
g^r(\sigma_1,\cdots,\sigma_n)
&=\frac{p_n+p_{n-1}}{q_n+q_{n-1}}-
  \frac{\big(\lfloor(N_0+n+1-2j)^{\frac{1}{\alpha}}\rfloor+1\big)p_n+p_{n-1}}
  {\big(\lfloor(N_0+n+1-2j)^{\frac{1}{\alpha}}\rfloor+1\big)q_n+q_{n-1}}  \\
&=\frac{\lfloor(N_0+n+1-2j)^{\frac{1}{\alpha}}\rfloor}
{\big((\lfloor(N_0+n+1-2j)^{\frac{1}{\alpha}}\rfloor+1)q_n+q_{n-1}\big)(q_n+q_{n-1})}{\color{red},}
\end{align*}

\begin{align*}
g^l(\sigma_1,\cdots,\sigma_n)
&=\frac{2\lfloor(N_0+n+1-2j)^{\frac{1}{\alpha}}\rfloor p_n+p_{n-1}}
{ 2\lfloor(N_0+n+1-2j)^{\frac{1}{\alpha}}\rfloor q_n+q_{n-1}}
   -\frac{p_n}{q_n} \\
&=\frac{1}{\big(2\lfloor(N_0+n+1-2j)^{\frac{1}{\alpha}}\rfloor q_n+q_{n-1}\Big)q_n}.
\end{align*}
Hence,
\begin{equation}
\label{eq4}
\hat{g}(\sigma_1,\cdots,\sigma_n)
=\frac{1}{\big(2\lfloor(N_0+n+1-2j)^{\frac{1}{\alpha}}\rfloor q_n+q_{n-1}\big)q_n}.
\end{equation}

  \item (ii) If $n=m_j$ or $n=m_j+1$ for some $j\geq 1$,
by the same analysis as above,
we have

%\begin{align*}
%g^r(b_1,\cdots,b_n)
%&=\frac{p_n+p_{n-1}}{q_n+q_{n-1}}-
%  \frac{(\lfloor n_j^{\lambda-1 }\rfloor+1)p_n+p_{n-1}}{(\lfloor n_j^{\lambda-1 }\rfloor+1)q_n+q_{n-1}}\\
%&=\frac{\lfloor n_j^{\lambda-1 }\rfloor}
%{\big((\lfloor n_j^{\lambda-1 }\rfloor+1)q_n+q_{n-1}\big)(q_n+q_{n-1})}.
%\end{align*}
%
%\[
%g^l(\sigma_1,\cdots,\sigma_n)
%=\frac{(2\lfloor n_j^{\lambda-1 }\rfloor)p_n+p_{n-1}}{(2\lfloor n_j^{\lambda-1 }\rfloor)q_n+q_{n-1}}
%   -\frac{p_n}{q_n}
%=\frac{1}{\big((2\lfloor n_j^{\lambda-1 }\rfloor)q_n+q_{n-1}\big)q_n}.
%\]
%Hence,
\begin{equation}
\label{eq5}
\hat{g}(\sigma_1,\cdots,\sigma_{m_j})
=\frac{1}{\big(2\lfloor n_j^{\lambda-1 }\rfloor q_{m_j}+q_{{m_j}-1}\big)q_{m_j}}{\color{red},}
\end{equation}
and
\begin{equation}
\label{eq6}
\hat{g}(\sigma_1,\cdots,\sigma_{m_j+1})
=\frac{1}{\big( 2\lfloor n_j^{\lambda^2-\lambda}\rfloor q_{m_j+1}+q_{m_j}\big)q_{m_j+1}}.
\end{equation}
\end{itemize}

%From \eqref{jnj}, \eqref{eq2}, \eqref{eq4}, \eqref{eq5},and \eqref{eq6},
%we have for any $n\geq 1$,
%\begin{equation}\label{eq7}
%\hat{g}(\sigma_1,\cdots,\sigma_n)
%\geq\frac{1}{2}|J_n(\sigma_1,\cdots,\sigma_n)|.
%\end{equation}

Let $n\geq1$ and $(\sigma_1,\cdots,\sigma_n)\in D_n(\alpha,\beta)$,
and for any $j\geq 1$, write
\begin{equation*}
g(\sigma_1,\cdots,\sigma_n):=\left\{
\begin{split}
&\frac{1}{\big(2\lfloor(N_0+n+1-2j)^{\frac{1}{\alpha}}\rfloor q_n+q_{n-1}\big)q_n},\,\,\,\,\,\,\,\,m_j+1<n<m_{j+1},\\
&\frac{1}{\big(2\lfloor n_j^{\lambda-1 }\rfloor q_{m_j}+q_{{m_j}-1}\big)q_{m_j}},\,\,\,\,\,\,\hspace{1.8cm}n=m_j,\\
&\frac{1}{\big(2\lfloor n_j^{\lambda^2-\lambda}\rfloor q_{m_j+1}+q_{m_j}\big)q_{m_j+1}},\hspace{1.7cm}n=m_j+1.
\end{split}\right.
\end{equation*}
Then we have
$$
g(\sigma_1,\cdots,\sigma_n)\geq |J_n(\sigma_1,\cdots,\sigma_n)|.
$$
From \eqref{eq4} and \eqref{eq5},
we know that the above definitions of $g(\sigma_1,\cdots,\sigma_n)$
gives a lower bound of the gaps on both sides of the fundamental interval $J_n(\sigma_1,\cdots,\sigma_n)$.

Fix $y\in A(\alpha,\beta)$ and take
\[r_0=\min_{1\leq k\leq m_{j_0}}\min_{(\sigma_1,\cdots,\sigma_k)\in D_k(\alpha,\beta)}
g(\sigma_1,\cdots,\sigma_k).\]
Then for any $0<r<r_0$,
there exists a unique sequence $\sigma_1,\sigma_2,\cdots$ such that $y\in J_j(\sigma_1,\cdots,\sigma_j)$
for all $j\geq 1$ and for some $n\geq m_{j_0+1}$,
$$
g(\sigma_1,\cdots,\sigma_{n+1})\leq r<g(\sigma_1,\cdots,\sigma_n).
$$
Thus the definition of $g(\sigma_1,\cdots,\sigma_n)$ implies that
the ball $B(y,r)$ can intersect only one fundamental interval of order $n$, that is $J_n(\sigma_1,\cdots,\sigma_n)$.\\
\indent Let us estimate the $\mu$-measure of each ball $B(y,r)$ with $y\in A(\alpha,\beta)$ and $r>0$.
If $m_j+1<n<m_{j+1}$ for any $j\geq1$,
then we deduce from \eqref{udy2} and \eqref{ineq-est-mea1} that
\begin{align}\label{q5}
\nonumber\mu(B(y,r))
&\leq \mu(J_n(\sigma_1,\cdots,\sigma_n))
=\lfloor(N_0+n+1-2j)^{\frac{1}{\alpha}}\rfloor\cdot\mu(J_{n+1}(\sigma_1,\cdots,\sigma_{n+1}))\\
\nonumber&\leq \lfloor(N_0+n+1-2j)^{\frac{1}{\alpha}}\rfloor\cdot|J_{n+1}(\sigma_1,\cdots,\sigma_{n+1})|^t\\
&\leq  \lfloor(N_0+n+1-2j)^{\frac{1}{\alpha}}\rfloor\cdot|g(\sigma_1,\cdots,\sigma_{n+1})|^t
\leq  r^{t-\varepsilon}.
\end{align}
If $n=m_j$ for some $j>j_0$, then we divide the case into two parts.\\
\indent (i) $r\leq |I_{m_j+1}(\sigma_1,\cdots,\sigma_{m_j+1})|$. In this case,
the ball $B(y,r)$ can intersect at most three basic intervals of order $m_j+1$,
which are $I_{m_j+1}(\sigma_1,\cdots,\sigma_{m_j+1}-1)$, $I_{m_j+1}(\sigma_1,\cdots,\sigma_{m_j+1})$
and $I_{m_j+1}(\sigma_1,\cdots,\sigma_{m_j+1}+1)$.
Then by \eqref{ineq-est-mea1},
\begin{align}\label{ineq5}
\nonumber\mu(B(y,r))&\leq3\mu(J_{m_j+1}(\sigma_1,\cdots,\sigma_{m_j+1}))
\leq 3|(J_{m_j+1}(\sigma_1,\cdots,\sigma_{m_j+1}))|^t\\
&\leq 3|g(\sigma_1,\cdots,\sigma_{m_j+1})|^t\leq 3r^t.
\end{align}

(ii) $r> |I_{m_j+1}(\sigma_1,\cdots,\sigma_{m_j+1})|$.
In this case, notice that
\[
|I_{m_j+1}(\sigma_1,\cdots,\sigma_{m_j+1})|
=\frac{1}{q_{m_j+1}(q_{m_j+1}+q_{m_j})}
\geq \frac{1}{8\lfloor n_j^{\lambda-1}\rfloor^2 q_{m_j}^2}{\color{red}.}
\]
Then, the number of fundamental intervals of order $m_j+1$ contained in $J_{m_j}(\sigma_1,\cdots,\sigma_{m_j})$,
intersecting with the ball $B(y,r)$,
is at most
$$
\frac{2r}{|I_{m_j+1}(\sigma_1,\cdots,\sigma_{m_j+1})|}
\leq 16r\cdot\lfloor n_j^{\lambda-1}\rfloor^2\cdot q_{m_j}^2+2
\leq 32 r\cdot\lfloor n_j^{\lambda-1}\rfloor^2q_{m_j}^2.
$$
Notice that $8r\cdot\lfloor n_j^{\lambda-1}\rfloor^2q_{m_j}^2>1$.
Then by \eqref{udy3} and \eqref{ineq-est-mea2},
we have
\begin{align}\label{gj7}
\mu(B(y,r))
\nonumber&\leq \min\Big\{\mu(J_{m_j}(\sigma_1,\cdots,\sigma_{m_j})),
             32 r\cdot\lfloor n_j^{\lambda-1}\rfloor^2q_{m_j}^2
            \mu(J_{m_j+1}(\sigma_1,\cdots,\sigma_{m_j+1}))\Big\}\\
\nonumber&\leq \mu(J_{m_j}(\sigma_1,\cdots,\sigma_{m_j}))
            \cdot\min\Big\{1, 32 r\cdot\lfloor n_j^{\lambda-1}\rfloor^2q_{m_j}^2
            \frac{1}{\lfloor n_j^{\lambda-1}\rfloor}\Big\}\\
\nonumber&\leq |(J_{m_j}(\sigma_1,\cdots,\sigma_{m_j}))|^t
            \cdot\min\Big\{1, 32 r\lfloor n_j^{\lambda-1}\rfloor q_{m_j}^2\Big \}\\
&\leq \left(\frac{1}{\lfloor n_j^{\lambda-1}\rfloor q_{m_j}^2}\right)^t
             \cdot 1^{1-t}\cdot (32 r\cdot\lfloor n_j^{\lambda-1}\rfloor q_{m_j}^2)^t\leq 32r^t,
\end{align}
where the last inequality is valid by using \eqref{eq2} and the elementary inequality
$\min\left\{a, b\right\}\leq a^{1-s}\cdot b^s$
which holds for any $a,b>0$ and $0<s<1$.

If $n=m_j+1$ for some $j>j_0$,
by the same analysis as that for the case $n=m_j$,
we still have
\begin{equation}
\label{gj8}
\mu(B(y,r))\leq 32r^t.
\end{equation}

Together with \eqref{q5}, \eqref{ineq5}, \eqref{gj7} and \eqref{gj8},
we deduce from Lemma \ref{mass principle}  that
\[\hdim A(\alpha,\beta)\geq t-\varepsilon. \]
Since $t<\frac{1}{\lambda+1}$ and $\varepsilon$ can be made arbitrarily small, we conclude that
\[
\hdim  A(\alpha,\beta)\geq \frac{1}{\lambda+1}=\frac{2}{\beta+2+\sqrt{\beta^2+4}}.
\]

\subsection{Lower bound for the case $\alpha=0$}

Using the same idea as in the case $0<\alpha<\infty$,
we shall construct a subset $A(0,\beta)$ of $E(0,\beta)$,
and then  use Lemma \ref{mass principle} to obtain the lower bound of the Hausdorff dimension of $E(0,\beta)$.
Since the computation is very similar,
here we only give the construction of the subset $A(0,\beta)$.

Let
\[F(0)=\big\{x\in(0,1):\ \lfloor e^{n}\rfloor+1\leq a_n(x)\leq\lfloor2e^{n}\rfloor,\ \ \forall\ n\geq1\big\}.\]
Then we insert the sequence $\{c_j\}_{j\geq1}$ and $\{e_j\}_{j\geq1}$
defined by \eqref{cj} and \eqref{ej} into the continued fraction expansion of $x=[a_1,a_2,\cdots]\in F(0)$
by putting
\begin{align*}
y=f(x)&=[b_1,b_2,\cdots,b_n,\cdots]\\
&=[a_1,a_2,\cdots, a_{r_1},c_1,e_1,a_{r_1+1},\cdots,a_{r_{j-1}},c_{j-1},e_{j-1},a_{r_{j-1}+1},\cdots],
\end{align*}
where $r_j$ is the smallest $r$ satisfying \eqref{zxrj}.
Let
\[
A(0,\beta):=\big\{y=f(x)\in(0,1)\colon x\in F(0)\big\}.
\]

By the condition \eqref{zxrj}, Lemma \ref{qnxx} and Stirling's formula,
we have
\[
  n_j>q_{r_j+2j-3}(a_1,a_2,\cdots, a_{r_1},c_1,e_1,a_{r_1+1},\cdots,a_{r_j-1})
 \geq\prod_{k=1}^{r_j-1}a_k\geq e^{\sum_{k=1}^{r_j-1}k}\geq e^{\frac{r^2_j}{4}}.
\]
Combining this inequality with Lemma \ref{sqn},
we have
\begin{align*}
&\ q_{r_j+2j-2}(a_1,a_2,\cdots, a_{r_1},c_1,e_1,a_{r_1+1},\cdots,a_{r_{j-1}},c_{j-1},e_{j-1},a_{r_{j-1}+1},\cdots,a_{r_j})\\
&\leq (a_{r_j}+1)n_j\leq4e^{2\sqrt{\log n_j}}\cdot n_j.
\end{align*}
Then we obtain $A(0,\beta)\subseteq E(0,\beta)$ by the same analysis as the case $0<\alpha<\infty$.
By following the proof of the above case step by step,
we can get the same conclusion.

{\bf Acknowledgement:}
The authors show their sincere appreciations to the anonymous referees for careful
reading and helpful comments.
Zhenliang Zhang is supported by project of
the Science and Technology Research Program of Chongqing Municipal Education Commission (No.KJQN202100528),
Natural Science Foundation of Chongqing (No.CSTB2022NSCQ-MSX1255).

\end{document}